\documentclass[reqno,tbtags]{amsproc}
\usepackage{tikz}

\newtheorem{theorem}{Theorem}[section]

\newtheorem{lemma}[theorem]{Lemma}
\newtheorem{corollary}[theorem]{Corollary}
\theoremstyle{definition}
\newtheorem{remark}[theorem]{Remark}
\newtheorem{remarks}[theorem]{Remarks}
\newtheorem{definition}[theorem]{Definition}

\newcommand \BDOTS {$~~\cdot~\cdot~\cdot~~$}
\newcommand \VPH   {\vphantom{$L^{L^L}$}}
\newcommand \Vph   {\vphantom{$L^L$}}
\newcommand \Vqh   {\vphantom{$L_q$}}

\def\op{\operatorname}
\def\subhead{\smallskip\noindent\textbf}
\def\ft{\footnote}
\def\bcdot{\cdot}

\usepackage[english,russian]{babel}
\usepackage{amsbib}

\begin{document}
\selectlanguage{english}

%\rpages{??--??}

%\udc{517.9}

\title[The absolute of finitely generated groups. II]{The absolute of finitely generated groups: II.~The Laplacian and degenerate part}
\thanks{Supported by the RSF grant 17-71-20153.}
\author{A.~M.~Vershik, A.~V.~Malyutin}
\address{St.~Petersburg Department of Steklov Institute of Mathematics; St.~Petersburg State University; Institute for Information Transmission Problems, Moscow, e-mail: {\tt avershik@gmail.com}.}
\address{St.~Petersburg Department of Steklov Institute of Mathematics; St.~Petersburg State University, e-mail: {\tt malyutin@pdmi.ras.ru}.}

\begin{abstract}
The article continues the series of papers on the absolute of finitely generated groups. The absolute of a group with a fixed system of generators is defined as the set of ergodic Markov measures for which the system of   \textit{cotransition probabilities} is the same as for the simple (right) random walk generated by the uniform distribution on the generators. The absolute is a~new boundary of a group, generated by random walks on the group.

We divide the absolute into the Laplacian part and degenerate part and describe the connection between the absolute, homogeneous Markov processes, and the Laplace operator; prove that the Laplacian part is preserved under taking some central extensions of groups; reduce the computation of the Laplacian part of the absolute of a nilpotent group to that of its abelianization; consider a number of fundamental examples (free groups, commutative groups, the discrete Heisenberg group).


\vskip2pt

{\Small DOI: {\tt https://doi.org/10.4213/faa3593}.}
\end{abstract}

\maketitle
\section{Introduction}

In this article, we continue to study the notion of the absolute of discrete groups and semigroups with a fixed system of generators, see the previous papers
\cite{V14b}--\cite{VM15}, \cite{VM17}, \cite{V14a},  \cite{VM18}. It appeared as a natural generalization of the well-known notions of  Poisson--Furstenberg boundary, Martin boundary, etc.\ for random walks on groups and semigroups and the exit boundary of the corresponding Markov chains. Unlike the Poisson--Furstenberg boundary, which is defined as a measure space, the absolute, like the Martin boundary,  is a topological boundary. On the other hand, the absolute of a group is a special case of the general notion of the absolute of a graded graph (or branching graph, or Bratteli diagram, which are different names for the same notion), i.e., the set of ergodic central measures on the Cantor-like set of infinite paths in  the so-called dynamic Cayley graph. All these terms are defined below, and here we explain the purpose of the suggested theory.

The theory of random walks on groups and the asymptotic theory of trajectories of these random walks have always been closely related to the theory of homogeneous Markov chains and harmonic analysis on groups, in particular, to the study of Laplace operators on groups, which are, in some sense or other, generators of Markov chains. The correspondence between the set of harmonic functions on a group with respect to a given Laplacian and the Poisson--Furstenberg (PF) boundary of the Markov chain is now well known and for the most part studied reasonably thoroughly (see \cite{V15}, \cite{D69}, \cite{GR85}--\cite{KW02}, and the references therein). However, this correspondence, for all its importance, reflects the asymptotic properties of the walk  only very roughly. First, for many groups the PF boundary, i.e., the set of harmonic positive functions, consists only of the constants; this is the case for Abelian, nilpotent, and some other groups. Second, the correspondence itself between the Laplacian and the asymptotic behavior of trajectories of the random walk is sometimes not defined, since for many Markov chains the notion of generator needs to be generalized and there is not yet an appropriate generalization of the traditional Laplace operator. Our key notion, that of the absolute of a random walk on a~group or semigroup with given generators, is defined not in terms of classes of functions on groups related to the Laplacian, but in terms of Markov measures (i.e., random walks on groups) having \textit{the same cotransition probabilities} as the simple random walk determined by the collection of generators. The PF boundary is only a part of the Laplacian absolute, namely, this is a space with a (harmonic) measure which realizes the decomposition of the Markov measure determined by the simple random walk into ergodic components. The whole absolute and, in particular, its Laplacian part contains much more information on the asymptotic properties of the group. E.~B.~Dynkin was probably the first to state, in some special cases, the problem of describing the Markov measures having the same cotransition probabilities as a~given Markov measure. 

We retain the term ``Laplacian absolute'' for a certain part of the absolute, namely, for the part consisting of nondegenerate measures; in this case, our generalization of the classical theory consists just in considering simultaneously not only harmonic functions, as usual, but also all minimal positive eigenfunctions of the Laplace operator. The main results of the paper concern primarily the Laplacian part of the absolute. In the relevant papers \cite{Mar66}, \cite{Mol67}, \cite{Wo00} we know of, this problem was posed both in the general setting \cite{Mol67}, and specifically for nilpotent groups (\cite{Mar66}, \cite{Wo00}). It turns out (see Theorem~\ref{th:ext-FC}) that one can give precise geometric conditions on an epimorphism of groups that guarantee the preservation of the Laplacian part of the absolute under taking the corresponding quotient. This provides a generalization and a new proof of the well-known Margulis' theorem on nilpotent groups~\cite{Mol67}. Theorem~\ref{th:ext-FC} implies that the Laplacian part of the absolute of a nilpotent group coincides with that of its abelianization.

The general problem of describing even the Laplacian part of the absolute of an arbitrary discrete finitely generated group is difficult, and it is not yet clear how wide is the class of groups for which it can be reduced to finding the PF boundary. We have demonstrated such a reduction for free groups~\cite{VM15}: in this case, the Laplacian part of the absolute is the product of the PF boundary and a~semiopen interval. We cite this result in the section containing examples. Note that the paper~\cite{VM18} is devoted to the problem of describing the absolute for commutative groups. Together with Theorem~\ref{th:ext-FC} of the present paper, this partially exhausts the problem in the case of nilpotent groups, but only for the Laplacian part of the absolute.

By definition, the degenerate part of the absolute consists  of the central ergodic measures for which some cylinder sets have zero measure. Another way to define degeneracy is to say that the support of the measure is a proper ideal in the path space of the graph (see below). A proper ideal of a dynamic Cayley graph may not be the dynamic graph of any group, hence the description of degenerate central measures in some way or other implicates finding the central measures for an arbitrary branching graph. Nevertheless, the study of the degenerate part of the absolute for the Heisenberg group shows that the approach from the viewpoint of geometric group theory proves to be useful also in this case: these problems are closely related to group-geometric problems (the theory of geodesics in groups), as well as to the theory of filtrations of $\sigma$-algebras and, of course, to harmonic analysis on groups.

The paper is organized as follows.

Section~2 contains numerous definitions and results related to the notion of absolute, but primarily those that are needed for what follows. For more details, see the literature cited above.

In Section~3 we study the correspondence between central measures and positive eigenfunctions of the Laplace operator. We prove (Lemma~\ref{cor:Markov-time}) that every measure from the Laplacian absolute is \textit{homogeneous} in an appropriate sense, and establish a bijective correspondence between the homogeneous nondegenerate central measures and the classes of proportional positive eigenfunctions of the Laplace operator (Claim\,2 of Theorem~\ref{th:measures-functions}). As a consequence, we obtain one of the main results of the paper: Claim\,1 of Theorem~\ref{th:measures-functions} saying that the measures from the Laplacian absolute are in a bijective correspondence with the classes of proportional minimal positive eigenfunctions of the Laplace operator.

In Section~4 we introduce the notion of a totally distorted subgroup and prove that taking the quotient of a group by a totally distorted subgroup whose all elements have finite conjugacy classes in the original group changes neither the collection of eigenfunctions of the Laplacian (Lemma~\ref{lem:eigenfunctions-invariant}), nor the Laplacian part of the absolute (Theorem~\ref{th:ext-FC}). Theorem~\ref{th:ext-FC} implies Corollary~\ref{cor:nilpotent} saying that the Laplacian part of the absolute of a nilpotent group coincides with that of its abelianization.

Section~5 contains results on the degenerate part of the absolute. The main result here is Theorem ~\ref{lem:degenerate-simple} saying that for every finitely generated group, for every  collection of generators, all degenerate central measures are concentrated on paths whose projections to the Cayley graph contain no cycles. It follows that in many interesting cases all degenerate central measures are concentrated on paths whose projections to the Cayley graph are geodesics (Corollary~\ref{cor:degenerate-geodesic}).

In Section~6 we give examples of describing the absolute, considering the cases of free groups, the discrete Heisenberg group, commutative and finite groups. For free groups and homogeneous trees,  the absolute  was described in our paper~\cite{VM15}. Here we briefly reproduce the main result, presenting its illustrative formulation. For the discrete Heisenberg group, we only state the result, its proof will be given in a~separate paper. The case of commutative groups was considered in~\cite{VM18}.

A detailed study of the absolute (including the degenerate part) for nilpotent and other classes of groups will be the subject of later papers by the authors.

\section{Necessary definitions}

We briefly give definitions of (old or new) notions used in what follows. For more detailed information on branching graphs and absolutes, not necessarily related to groups, see the papers cited in the introduction.

\subhead{2.1. Graphs, branching graphs, the Cayley and dynamic graph of a~group.}
By a \textit{graph} we mean a locally finite directed graph with a marked vertex. Loops and multiple edges are allowed. A \textit{path} in a graph is a (finite or infinite) sequence of alternating vertices and edges of the form
$$
v_0, e_1, v_1, e_2, \dots, e_n, v_n
$$
where $e_k$ is an edge with initial vertex~$v_{k-1}$ and terminal vertex~$v_k$ (both vertices and edges may repeat). We consider graphs in which every path can be extended to an infinite path. A special class of graphs is that of branching graphs. A \textit{branching graph} is a graph in which the set of paths from the marked vertex~$v_0$ to every vertex~$v$ is nonempty (in this case, one says that $v$ is  \textit{reachable} from~$v_0$) and all these paths have the same length. On the set of vertices of a branching graph there is a natural grading by the distance to  the marked vertex. Such graphs are also called (locally finite) \textit{${\mathbb N}$-graded graphs} or \textit{Bratteli diagrams}.

A graph~$\Gamma$ with a marked vertex~$v_0$ gives rise in a canonical way to the \textit{dynamic graph} $\op{D}_{v_0}(\Gamma)$, a branching graph constructed as follows. The $n$th level of $\op{D}_{v_0}(\Gamma)$ is a copy of the set of vertices of~$\Gamma$ connected with the marked vertex~$v_0$ by paths of length~$n$. Given vertices $v_1$ and $v_2$ in~$\op{D}_{v_0}(\Gamma)$, they are connected by exactly $k$ edges directed from~$v_1$ to~$v_2$  if and only if $v_2$ lies in the next level of~$\op{D}_{v_0}(\Gamma)$ as compared to~$v_1$, and the vertex~$w_1$ of~$\Gamma$ that corresponds to~$v_1$ is connected with the vertex~$w_2$ of~$\Gamma$ that corresponds to~$v_2$ by exactly $k$ edges. Note that a graph~$\Gamma$ coincides with its dynamic graph~$\op{D}_{v_0}(\Gamma)$ if and only if $\Gamma$ is a branching graph.

Let us apply the above definitions to groups.  The notion of the Cayley graph of a (semi)group with a chosen collection of generators is well known. By a collection of generators in a group~$G$ we mean a subset of~$G$ generating it as a semigroup. It is not necessarily symmetric and does not necessarily contain the identity. Note that the results of this paper extend automatically to the case of collections of generators with multiplicities or weights. The \textit{dynamic {\rm(}Cayley{\rm)} graph} of a group~$G$ with a~chosen collection of generators~$S$, in what follows denoted by $\op{D}(G,S)$, is the dynamic graph (in the sense of the above definition) constructed from the Cayley graph of the pair~$(G,S)$.

\subhead{2.2. The path space of a branching graph, central measures, ergodicity, absolute.}
Let $\op{D}$ be a branching graph. Denote by $T(\op{D})$ the set of all infinite paths in~$\op{D}$ from the marked vertex. It is equipped with the weak (projective limit) topology, defined in a natural way, in which  $T(\op{D})$ is compact. Denote by~$\mathcal{M}(\op{D})$ the set of Borel probability measures on this space. A measure~$\nu$ on~$T(\op{D})$ is called \textit{central} if for almost every (with respect to~$\nu$) path~$t$, the conditional measure on the set of paths that differ from~$t$ in finitely many places is uniform. In another terminology, this means that the tail equivalence relation is  \textit{semihomogeneous}. An equivalent definition of centrality is as follows: a measure on~$T(\op{D})$ is central if for every vertex~$v$ of~$\op{D}$, the probabilities of all finite paths from the marked vertex to~$v$ corresponding to this measure are equal. One can easily show that all central measures are Markovian. The central measures constitute a convex compactum~$\mathcal{C}(\op{D})$, which is a simplex (see~\cite{V15}) in~$\mathcal{M}(\op{D})$. The simplex~$\mathcal{C}(\op{D})$ is the projective limit of a sequence of finite-dimensional simplices of measures on finite paths; see~\cite{V15}. Here we do not use this fact.

A central measure is called \textit{ergodic} (or \textit{regular}) if it is an extreme point of the simplex~$\mathcal{C}(\op{D})$. The \textit{absolute} of a branching graph is the set of all ergodic central measures on the compactum~$T(\op{D})$ of infinite paths starting at the marked point.

Now we will apply all these notions to the dynamic graph of a group.

\begin{definition}
The \textit{absolute of a finitely generated group with a fixed finite collection of generators} is the absolute of the corresponding dynamic graph.
\end{definition}

The above definitions are equivalent to the following definition in terms close to the theory of random walks. The absolute of a group~$G$ is the set of Markov measures generated by ergodic random walks on the Cayley graph of~$G$ and satisfying the following property: for every~$n$ and for every element of~$G$ representable as a~product of $n$ generators, the conditional measure on its representations as a~product of $n$ generators is uniform. For stationary Markov measures, the notion of a~central measure coincides with the well-known notion of a measure of maximal entropy. Note that the ergodicity of a measure, defined above in terms of the impossibility of writing it as a nontrivial convex combination of other central measures, can be defined directly, in terms of the intersection of the $\sigma$-algebras of pasts (see Section~2.5).

The absolute of a group~$G$ with a collection of generators~$S$ is denoted by~$\mathcal{A}(G,S)$.

\subhead{2.3. The Laplace operator and its eigenfunctions.}
Given a group $G$ and a~collection of generators~$S$, we define the \textit{Laplace operator} acting on the space of all functions on~$G$ as the linear operator given by the formula
\begin{equation*}
(\Delta_Sf)(g):=\frac{\sum_{s\in S}f(gs)}{|S|}\,.
\end{equation*}

We will be interested in the action of~$\Delta_S$ on nonnegative functions and its eigenfunctions corresponding to nonnegative eigenvalues.

A nonnegative eigenfunction~$f$ of an operator~$A$ is called \textit{minimal} if every nonnegative eigenfunction of~$A$ with the same eigenvalue dominated by~$f$ is proportional to~$f$.

The relationships between the operator~$\Delta_S$, its eigenfunctions, Markov measures, and random walks on groups will be studied in more detail in the next section.

\subhead{2.4. The nondegenerate and degenerate parts of the absolute.}
A measure~$\nu$ on the path space of a branching graph will be called \textit{nondegenerate} if the probability of every finite path (i.e., every cylinder set) is nonzero. The \textit{main part} of the absolute is its subset  consisting of nondegenerate measures. In the group case, the main part of the absolute is called the  \textit{Laplacian part}, or simply \textit{Laplacian absolute}. The set of degenerate ergodic central measures will be called the \textit{degenerate part}  of the absolute. It is related to the geometry of groups and the geometry of paths in Cayley graphs.

Every graded graph determines a partial order on its vertices. An \textit{ideal} of this partial order is a subset of vertices that contains, along with every vertex~$v$, all vertices smaller than~$v$. One can consider the path space of a~given ideal, as a~closed subset in the space of all paths. It is easy to see that the definition of degenerate and nondegenerate central measures can be reformulated using the notion of ideal as follows. The support of every central measure (as a closed set in the path space) is always the space of all paths of some ideal. If this ideal coincides with the whole path space, then the measure is nondegenerate; if the ideal is proper, then the measure is degenerate. Hence in order to find the whole absolute, one has to find the absolutes of some proper subgraphs of the dynamic graph. The problem of describing the ideals for subgraphs of a given graph is of independent interest.

\subhead{2.5. Additional structures and comments.}

\smallskip
\textbf{A filtration structure.} We consider the tail equivalence relation on the path space of a graph (for instance, a branching graph). Here it is convenient to use the framework of the theory of filtrations (see~\cite{V17}). The \textit{tail filtration on the path space of a graph}, or, in short, the \textit{tail filtration}, is defined as follows. Two paths are said to be \textit{$n$-equivalent} if they coincide from the $n$th vertex, and  \textit{tail equivalent} if they are $n$-equivalent for some~$n$. Let $\mathfrak{A}_n$, $n\in{\mathbb N}_0$, be the $\sigma$-algebra of Borel subsets in the compactum of infinite paths that contain, along with every path, all $n$-equivalent paths. The decreasing sequence 
$$
\mathfrak{A}_0 \supset \mathfrak{A}_1 \supset \mathfrak{A}_2 \supset \dots
$$
 of $\sigma$-subalgebras of Borel sets is called the \textit{tail filtration on the path space}. The equivalence classes described above are also called  \textit{blocks}, or \textit{elements}, of the filtration. The filtration is called \textit{ergodic} if the intersection of the above $\sigma$-algebras is trivial. An \textit{automorphism of the filtration} is an automorphism of the measure space preserving the tail equivalence classes. In these terms, a \textit{central measure} is a~measure invariant under all automorphisms of the tail filtration.

\smallskip
\textbf{Generalizations: equipped multigraphs.} In the present paper, we consider primarily simple walks related to the uniform distribution on the collection of generators, but one can also consider the absolute in a more general setting depending on a cocycle (for more details on equipped multigraphs and their cocycles, see Sections~3.3 and~3.5 in~\cite{V17}). The generalization to the case of multigraphs will be needed in Section~6.2.

\smallskip
\textbf{The compactum of infinite words.}
An alternative approach to the theory of absolutes of groups and semigroups is provided by an algebraic structure: the dynamic graph of a group is the Cayley graph of a graded semigroup, and the path space of the dynamic graph is canonically isomorphic to the space of infinite words in the alphabet of generators (see~\cite{Eetal92}). In this paper, we use this approach in the example from Section~5 and in Section~6.4 about the discrete Heisenberg group. Within this approach, when considering finite and infinite sequences of elements of a set, we say that this set is an \textit{alphabet}, its elements are \textit{symbols}, or \textit{letters}, and sequences of elements are  \textit{words} over this alphabet. For presentations of groups, we use alphabets consisting of pairs of letters of the form~$\{a, a^{-1}\}$; such letters are called \textit{inverse} to each other. A \textit{word without inverse letters} is a word in which no two letters are inverse (not only neighboring letters are taken into account). Words are written without commas, a block of $n$ successive letters~$a$ is denoted by~$a^n$, a block of $n$ successive letters~$a^{-1}$ is denoted by~$a^{-n}$. The notation $(a^{-1})^{-1}$ is interpreted as~$a$,  and $a^0$ stands for the empty (sub)word. The \textit{inverse} word to $w_1\dots w_n$ is $w_n^{-1}\dots w_1^{-1}$. The notation~$a^{+\infty}$ (respectively, $a^{-\infty}$) corresponds to the right-infinite word whose every letter is~$a$ (respectively, $a^{-1}$).

\section{Homogeneous measures, Laplace operator, and Laplacian absolute}

In this section we establish a connection between homogeneous Markov chains on a group~$G$ (more exactly, on the Cayley and dynamic graphs) 
and functions associated with the Laplace operator on~$G$. This connection is well known for harmonic functions, i.e., functions invariant under the Laplace operator; it is the basis of harmonic analysis on groups. For eigenfunctions, the connection is less known and less researched; perhaps, one of the first papers on the subject, using a somewhat different viewpoint, was S.~A.~Molchanov's paper~\cite{Mol67}. We treat the subject systematically and refine the correspondence between properties of eigenfunctions and their linear combinations (positivity, minimality, etc.)\ on the one hand, and properties of Markov chains (homogeneity, centrality, ergodicity, etc.)\ on the other hand. It is this correspondence that provides a link between the nondegenerate part of the absolute and the theory of the Laplace operator. Note that it is the absence of such a correspondence for degenerate Markov chains that causes difficulties in the study of the degenerate part of the absolute.

Before proving the main theorem, we discuss the notion of homogeneity. Let $G$ be a finitely generated group and $S$ be a finite collection of generators of~$G$. We construct the dynamic graph~$\op{D}(G,S)$ corresponding to the pair~$(G,S)$ and consider central measures on the compactum of infinite paths of this graph. It is easy to verify (see the introduction) that the random processes corresponding to central measures, both on the dynamic graph and on the Cayley graph, are Markovian.

\begin{definition}
Measures on the path space of the dynamic graph that give rise to time-homogeneous Markov chains on the Cayley graph will be called \textit{homogeneous}. Recall that a Markov chain is called \textit{time-homogeneous} if its transition probabilities do not depend on time. (This terminology causes no ambiguity, since there are no time-inhomogeneous Markov processes on the dynamic graph.)
\end{definition}

\begin{lemma}[homogeneity lemma]
\label{cor:Markov-time}
All points of the Laplacian absolute, i.e., all nondegenerate ergodic central measures on the path space of the dynamic graph are homogeneous. In other words, every such measure is generated by a homogeneous Markov measure on the group.
\end{lemma}

\begin{proof}
Let $v$ be a vertex of the dynamic graph~$\op{D}=\op{D}(G,S)$ and
  $\op{D}_v$ be the subgraph in~$\op{D}$ formed by the vertices and edges reachable from~$v$. If a central measure~$\mu$ does not vanish on the set of paths passing through~$v$ (for example, if $\mu$ is nondegenerate), then we can consider its restriction $\mu_v$ to the path space of~$\op{D}_v$. The group structure gives rise to a \textit{translation}, a~canonical natural isomorphism~$\alpha$ between~$\op{D}_v$ and~$\op{D}$. If $v$ represents a central element (for example, the identity) of the group, then for every vertex~$x$ in~$\op{D}$, the vertex~$\alpha^{-1}(x)$ is, obviously, reachable from~$x$. It follows that the measure~$\mu$ dominates the measure~$\alpha_*(\mu_v)$. Clearly, the measures~$\mu_v$ and~$\alpha_*(\mu_v)$ are central. Therefore, if  $\mu$ is ergodic, then $\alpha_*(\mu_v)=\mu$.

Applying the above arguments to the vertices of the dynamic graph that project to the identity of the group, we see that the transition probabilities of the Markov chain corresponding to a nondegenerate ergodic central measure are invariant under the time shift by~$k$ provided that the Cayley graph of the pair~$(G,S)$ contains a~cycle of length~$k$. However, in the group case, the Markov chain corresponding to a nondegenerate central measure is  \textit{indecomposable} in the sense of the theory of Markov chains, and it is well known (see, for example, Theorem~3.3 in~\cite{ChKL64}) that in an indecomposable chain, the time shift by the period of the chain can be written as the composition of the (direct and inverse) shifts by the  lengths of some cycles, which implies the desired result.
\end{proof}

\begin{remark}
The second claim of Theorem~\ref{th:measures-functions} proved below shows that usually only a small part of central measures are homogeneous.
\end{remark}

Now we proceed to a detailed description of the relationships between central measures and the Laplace operator. In particular, we will explain the connection between this operator and the Laplacian absolute. A Markov measure on paths of the dynamic graph is determined by a collection of transition probabilities. Our aim is, given the eigenfunctions of the Laplace operator, to construct a system of transition probabilities, i.e., to define some measure. Given a positive function $f\colon G\to{\mathbb R}$, we define the transition probability~$p(g,gs)$, where $g\in G$ and $s\in S$, by the rule
\begin{equation}
\label{eq:nuf-gen}
p(g,gs):=f(gs)\bigg/\sum_{t\in S}f(gt).
\end{equation}
Denote by~$\mathcal{D}$ the mapping that sends a function~$f$ to the collection of transition probabilities on the Cayley graph given by~\eqref{eq:nuf-gen}, and thus to a Markov measure~$\nu_f$ on the path space of the dynamic graph. Recall that the classical theory establishes a correspondence between the minimal positive harmonic functions and the ergodic central measures of a certain form, or the points of the Poisson--Furstenberg boundary. The following theorem extends this correspondence.

\begin{theorem}
\label{th:measures-functions}
{\rm1.} The mapping~$\mathcal{D}$ induces a bijection between the set of classes of proportional positive \underline{minimal} eigenfunctions of the Laplacian and the Laplacian absolute, i.e., the set of all nondegenerate central \underline{ergodic} measures on the path space of the graph~$\op{D}(G,S)$.

{\rm2.} The mapping~$\mathcal{D}$ induces a bijection between the set of classes of proportional positive {\rm(}not necessarily minimal{\rm)} eigenfunctions of the Laplacian and the set of all \underline{homogeneous} nondegenerate central (not necessarily ergodic{\rm)} measures on the path space of the graph~$\op{D}(G,S)$.
\end{theorem}

The first claim of Theorem~\ref{th:measures-functions} is the main result of this section.

To prove Theorem~\ref{th:measures-functions}, we will need the following lemma, which introduces the notion of  \textit{characteristic}.

\begin{lemma}[on the existence of the characteristic]\label{lem:homogen-2}
Let $\nu$ be a homogeneous nondegenerate central measure on the path space of the graph~$\op{D}(G,S)$. Then there exist a unique number  $A_\nu \in {\mathbb R}$ (called the \textup{characteristic} of~$\nu$) and a unique function $f_\nu\colon G\to{\mathbb R}$ such that for every finite path~$P$ in the dynamic graph, 
\begin{equation}
\label{eq:f4}
\nu(P)=f_\nu(g)\cdot A_\nu^{|P|},
\end{equation}
where $\nu(P)$ is the measure of the cylinder set of paths that begin with~$P$,
$|P|$ is the length of $P$, and $g$ is the element of~$G$ corresponding to the terminal vertex of~$P$.
\end{lemma}

\begin{proof}
Choose an arbitrary path~$Z$  in $\op{D}(G,S)$ of nonzero length~$|Z|$ that represents the identity of~$G$, and put $A_\nu:=\nu(Z)^{1/|Z|}$. The value $A_\nu$ does not depend on the choice of~$Z$. Indeed, let $Y$ be another path representing the identity. Denote by~$Z^{|Y|}$ the path whose projection to the Cayley graph is the 
$|Y|$-fold repetition of the projection of~$Z$. Then, since the paths~$Z^{|Y|}$ and~$Y^{|Z|}$ have the same length, the centrality implies that
\begin{equation}
\label{eq:YZ=ZY}
\nu(Z)^{|Y|}=\nu(Z^{|Y|})\stackrel{\text{centrality}}{=}\nu(Y^{|Z|})=\nu(Y)^{|Z|}.
\end{equation}

Further, for each element $g\in G$ choose an arbitrary path~$P_g$ representing this element in the dynamic graph and put
\begin{equation}
\label{eq:fnu}
f_{\nu}(g):=\nu(P_g)\cdot A_\nu^{-|P_g|}.
\end{equation}
Then $f_{\nu}$ does not depend on the choice of a path: if for some path $Q_g$ representing the same element~$g$, the values $\nu(Q_g)\cdot A_\nu^{-|Q_g|}$ and $\nu(P_g)\cdot A_\nu^{-|P_g|}$ differed, then, extending~$P_g$ and~$Q_g$ by paths with the same projections to the Cayley graph to paths representing the identity, we would obtain a contradiction with~\eqref{eq:YZ=ZY}.

To see the uniqueness of $A_\nu$, it suffices to consider the paths representing the identity of~$G$. By~\eqref{eq:f4}, this immediately implies the uniqueness of~$f_\nu$.
\end{proof}

\begin{remarks}
1. The notion of characteristic is parallel to the notion of eigenvalue for eigenfunctions of the Laplacian on the semigroup ${\mathbb N}_0:=\{0,1,2,\dots\}$; to see this, regard the dynamic graph as a subset in the direct product of the Cayley graph of the group and the Cayley graph of the semigroup ${\mathbb N}_0:=\{0,1,2,\dots\}$ and use the canonical bijection between the Laplacian absolute and the set of harmonic functions on the dynamic graph: harmonic functions corresponding to homogeneous measures can be decomposed into the product of eigenfunctions of the Laplacian on the original group and on the semigroup ${\mathbb N}_0$. (The problem of decomposing minimal harmonic functions into such products is discussed in~\cite{Mol67}.)

2. It is clear from the proof of Lemma~\ref{lem:homogen-2} that (in terms of this lemma)
 $\log A_\nu$ is given by the formula
\begin{equation}
\label{eq:char}
\frac{\log\nu(P_1)-\log \nu(P_2)}{|P_1|-|P_2|},
\end{equation}
where~$P_1$ and~$P_2$ are arbitrary paths of different lengths in the Cayley graph for which both initial and terminal vertices coincide,  $|P_i|$ is the length of $P_i$, and $\nu(P_i)$ is the product of the transition probabilities corresponding to the measure~$\nu$ over all edges of~$P_i$ {\rm(}counting multiplicities{\rm)}.
\end{remarks}

\begin{proof}[Proof of Claim~2 of Theorem~\ref{th:measures-functions}]
If $f$ is an eigenfunction of the Laplacian with eigenvalue~$\alpha$, then the defining formula~\eqref{eq:nuf-gen} implies that the value of the measure  $\nu_f=\mathcal{D}(f)$ on the cylinder set of paths that begin with a given path~$P$ of length~$k$ representing an element~$g$ is given by the formula
\begin{equation}
\label{def:check-eigen}
\nu_f(P)=\frac{f(g)}{f(1_G)}\cdot(\alpha\cdot |S|)^{-k}.
\end{equation}
This proves the centrality of~$\nu_f$, and its homogeneity follows from the definition.

Conversely, let $\nu$ be a homogeneous nondegenerate central measure on the path space of the graph~$\op{D}(G,S)$. Then the corresponding function~$f_\nu$ from Lemma~\ref{lem:homogen-2} is an eigenfunction of the Laplacian with eigenvalue $A_\nu^{-1}\cdot |S|^{-1}$ (this follows from~\eqref{eq:fnu}). Comparing formulas~\eqref{eq:fnu} and~\eqref{def:check-eigen}, we see that the composition of the correspondences $f\mapsto \nu_f$ and $\nu\mapsto f_\nu$ are identity mappings on the sets under considerations.
\end{proof}

\begin{proof}[Proof of Claim~1 of Theorem~\ref{th:measures-functions}]
Let us show that the bijection from the second claim of the theorem proved above sends ergodic measures to minimal functions. (This will immediately imply the desired result by the homogeneity lemma~\ref{cor:Markov-time}.) One can see from formula~\eqref{eq:fnu} defining the function~$f_\nu$ that the correspondence $\nu\mapsto f_\nu$ sends the simplex of measures with characteristic~$A$ to the simplex of positive eigenfunctions (taking the value~$1$ at the identity of the group) of the Laplacian with eigenvalue~$A^{-1}\cdot |S|^{-1}$ preserving the affine structure. Since the minimality of an eigenfunction is defined with respect to functions having the same eigenvalue, it follows that the correspondence $\nu\mapsto f_\nu$ yields an embedding of the main part of the absolute into the set of classes of proportional minimal positive eigenfunctions of the Laplacian. To check that this embedding is bijective, one should verify that the decomposition of a homogeneous nondegenerate central measure~$\nu$ into ergodic components involves neither degenerate measures, no measures with characteristics different from that of~$\nu$. To see this, in the Cayley graph choose a cycle of nonzero length with endpoints at the identity of the group and observe that, by~\eqref{eq:char}, the values of homogeneous measures with different characteristics  on the sequence of powers of this cycle produce exponentials with different bases, while every degenerate measure vanishes on this sequence (see Theorem~\ref{lem:degenerate-simple}).
\end{proof}

\begin{remarks}
1 (on the algebra of eigenfunctions). Apparently, the Laplacian absolute can be described and studied by analogy with the theory of harmonic functions, in which one introduces a nontrivial commutative multiplication of such functions, whereupon, on the one hand,  the Gelfand spectrum of the corresponding Banach algebra coincides with the set of minimal positive harmonic functions and, on the other hand, this is exactly the exit boundary of the simple walk. Such a construction of the Laplacian absolute would, first, clarify the situation and, second, simplify the proofs given above. (See~\cite{KV83}, \cite{Fo89} and the literature on harmonic analysis.) It is also of interest to establish direct analytic links between the asymptotics of typical trajectories of ergodic random walks on a group and the corresponding positive minimal eigenfunctions of the Laplace operator.

2 (on minimal eigenfunctions). The results of this section hold in a more general context: in particular, the main part of the absolute can be described in terms of minimal eigenfunctions of the Laplacian in the case of a wide class of graphs that do not have symmetries of Cayley graphs of groups (see~\cite{Mol67}). Besides, note that the collection of minimal eigenfunctions itself was studied in the literature and is well known for some cases (see~\cite{Wo00}, as well as a description of minimal eigenfunctions for the product of graphs in~\cite{Mol67}).

3 (on the structure of the Laplacian absolute). In the investigated cases, there is a bijection between the spaces of positive minimal eigenfunctions on a group for any two nonextreme eigenvalues. Accordingly, in these cases the Laplacian absolute (with the ``extreme'' fiber excluded) is the product of some space and an interval. The authors do not know examples of groups for which this property does not hold. In any case, the class of groups for which it does hold is of interest.

4 (on the group action on the Laplacian absolute). Identifying the Laplacian absolute with the set of minimal eigenfunctions of the Laplacian (Theorem~\ref{th:measures-functions}), we obtain an action of the group on its Laplacian absolute. It is of interest to describe the class of groups for which this action is trivial. This is exactly the class of groups for which every nondegenerate ergodic central measure gives rise to a Markov chain with independent identically distributed increments (the transition probabilities are the same at all vertices of the graph and depend only on the generators labeling the edges). This class contains all finitely generated nilpotent groups (see Section~6.3 and Corollary~\ref{cor:nilpotent} below).
\end{remarks}

\section{Preservation of the Laplacian absolute under extensions}

In this section we introduce the notion of a totally distorted subgroup and prove that taking the quotient of a group by a totally distorted subgroup whose all elements have finite conjugacy classes in the original group changes neither the collection of eigenfunctions of the Laplacian (Lemma~\ref{lem:eigenfunctions-invariant}), nor the Laplacian part of the absolute (Theorem~\ref{th:ext-FC}). These results are a natural generalization of Margulis' theorem, which arises when looking on the latter from the viewpoint of geometric group theory. A crucial role here is played by the notion of a distorted subgroup from this theory. As we will see, in the case of nilpotent groups the existence of a~central totally distorted subgroup is a typical situation. This gives Corollary~\ref{cor:nilpotent} saying that the Laplacian absolute of a nilpotent group coincides with that of its abelianization.

\begin{definition}
A finitely generated subgroup~$H$ of a finitely generated group~$G$ is called \textit{distorted} if the identity embedding $H\to G$ is not a quasi-isometry with respect to word metrics; the property of being distorted does not depend on the choice of a system of generators (see~\cite{Gr93}). A subgroup~$K$ in a group~$G$ will be called \textit{totally distorted} if all infinite cyclic subgroups of~$K$ are distorted in~$G$. \end{definition}

In the case of an infinite cyclic subgroup, we have the following equivalent definition of distortion: the infinite cyclic subgroup  $\langle g \rangle$ generated by an element $g\in G$ is distorted if and only if $|g^k|_G=o(k)$, where $|\bcdot|_G$ is the length of an element in the word metric of the group~$G$. Thus, a subgroup~$K$ in $G$ is totally distorted if and only if $|g^k|_G=o(k)$ for every element $g\in K$.

\begin{theorem}\label{th:ext-FC}
Let $\phi\colon G\to G/K$ be an epimorphism of finitely generated groups such that the kernel~$K$ is a totally distorted subgroup and every element of~$K$ has a finite conjugacy class in~$G$. Let $S$ be a collection of generators of~$G$, and assume that  $\phi$ is injective on~$S$. There arises a natural isomorphism~$\phi_*$ between the path spaces of the graphs~$\op{D}(G,S)$ and~$D(G/K,\phi(S))$. Then $\phi_*$ induces an isomorphism of the Laplacian absolutes of the pairs~$(G,S)$ and~$(G/K,\phi(S))$.
\end{theorem}

The main fact used in the proof of Theorem~\ref{th:ext-FC} is that $\phi_*$ sends central measures to central measures. If $K$ is not totally distorted, then the latter may not be true. To prove Theorem~\ref{th:ext-FC}, we will need the following lemma related to Margulis' theorem on harmonic functions on nilpotent groups~\cite{Mar66} (see also~\cite{Wo00}). This lemma can be directly extended to Laplacians of arbitrary measures whose supports generate the whole group; however, here we restrict ourselves to  the main case we are interested in.

\begin{lemma}\label{lem:eigenfunctions-invariant}
Let $G$ be a finitely generated group. Assume that a subgroup~$H$ in~$G$ is totally distorted and every element of~$K$ has a finite conjugacy class in~$G$. Let $S$ be a collection of generators in~$G$. Then every positive eigenfunction of the Laplacian constructed from~$S$ is constant on the right cosets of~$H$.
\end{lemma}

\begin{proof}
Consider a left-invariant word metric~$d$ on~$G$ (for example, the one corresponding to the collection of generators~$S$), and associate with an element $g$ of~$G$ the function
$$
\phi_g\colon G\to{\mathbb R},\qquad x\mapsto d(x,gx)=d(1_G, x^{-1}gx),
$$
sending an element $x\in G$ to the distance $d(x,gx)$ (in the chosen metric) between~$x$ and~$gx$. If the conjugacy class of~$g$ in~$G$ is finite, then, obviously, the function~$\phi_g$ takes only finitely many values (since the element~$x^{-1}gx$ is conjugate to~$g$). Then we deduce, using Harnack's inequality (see, e.g.,~\cite[p.~262]{Wo00}), that for every positive eigenfunction~$f$ of the Laplacian there is a positive constant~$\varepsilon$ such that ${f(x)>\varepsilon\cdot f(gx)}$ for all~$x\in G$. If $f$ is minimal and thus cannot dominate eigenfunctions not proportional to~$f$, then it follows that for some positive $t\in{\mathbb R}$ we have $f(gx)=t\cdot f(x)$ for all $x\in G$. But $f(g^n)=t^n\cdot f(1_G)$, hence if the subgroup~$\langle g\rangle$ generated by~$g$ is finite or distorted, then, again using Harnack's inequality and applying it this time to elements of the sequence~$f(g^n)$, we see that $t=1$. It remains to observe that every eigenfunction of the Laplacian can be decomposed into a sum of minimal eigenfunctions.
\end{proof}

\begin{proof}[Proof of Theorem~\ref{th:ext-FC}]
Observe that for an arbitrary group epimorphism $G_1\to G_2$ preserving the collection of generators, every eigenfunction of the Laplacian on~$G_2$ can be lifted to an eigenfunction of the Laplacian on~$G_1$, but in the general case the Laplacian on~$G_1$ can have eigenfunctions that do not correspond to any eigenfunctions of the Laplacian on~$G_2$. If the kernel of the epimorphism satisfies the conditions of Lemma~\ref{lem:eigenfunctions-invariant}, then the lemma guarantees that every positive eigenfunction of the Laplacian on~$G_2$ is the lifting of an eigenfunction of the Laplacian on~$G_1$. Thus, an epimorphism with such a kernel induces a bijection at the level of positive eigenfunctions of the Laplacians. The proof is completed by passing from eigenfunctions of the Laplacians to the Laplacian absolutes (Theorem~\ref{th:measures-functions}).
\end{proof}

\begin{corollary}
\label{cor:nilpotent}
Let $N$ be a finitely generated nilpotent group and ${\phi\colon N\to \operatorname{Ab}(N)}$ be the abelianization homomorphism. Let $S$ be a collection of generators in~$N$, and assume that $\phi$ is injective on~$S$. There arises a natural isomorphism~$\phi_*$ between the path spaces of the graphs $\op{D}(N,S)$ and $\op{D}(\operatorname{Ab}(N),\phi(S))$. Then $\phi_*$ induces an isomorphism of the Laplacian absolutes of the pairs~$(N,S)$ and $(\operatorname{Ab}(N),\phi(S))$.
\end{corollary}

\begin{proof}
In an arbitrary group, the subgroup generated by the commutators lying in the center of the group is totally distorted; this can easily be deduced from the fact that if a commutator~$[a,b]$ lies in the center, then $[a^m,b^n]=[a,b]^{mn}$ for arbitrary integers~$m$ and~$n$.

In a nilpotent group~$N$ of class~$s$, the subgroup~$N_{s-1}$ (where $N_i=[N,N_{i-1}]$,   $N_0:=N$) is generated by the commutators lying in the center of~$N$ and, consequently, totally distorted, the quotient~$N/N_{s-1}$ being a nilpotent group of class~$s-1$, with $\operatorname{Ab}(N)=\operatorname{Ab}(N/N_{s-1})$. Therefore,  $\operatorname{Ab}(N)$ is obtained from~$N$ by taking a~sequence of quotients by central totally distorted subgroups, and the claim follows from Theorem~\ref{th:ext-FC}.
\end{proof}

\begin{remark}
Theorem~\ref{th:ext-FC} and Lemma~\ref{lem:eigenfunctions-invariant} generalize, in the sense of significantly extending the class of groups for which the described phenomenon holds, the important Margulis' theorem~\cite{Mar66} saying that every positive eigenfunction of the Laplacian of a nilpotent group is constant on the cosets of the commutant. See also~\cite{Wo00}.
\end{remark}

\section{The degenerate part of the absolute and geodesics on the group}

In this section we prove a number of results concerning the degenerate part of the absolute. As in the previous sections, let $G$ be a finitely generated group with a finite collection of generators~$S$. We construct the dynamic graph~$\op{D}(G,S)$ corresponding to the pair~$(G,S)$ and consider central measures on the space of infinite paths of this graph.

\begin{theorem}\label{lem:degenerate-simple}
All degenerate central measures on the path space of the dynamic graph are concentrated on paths whose projections to the Cayley graph do not contain cycles.
\end{theorem}

\begin{proof}
It suffices to consider the case of an ergodic (degenerate central) measure. Assume that the probability of passing through a finite path~$P$ whose projection to the Cayley graph contains a cycle does not vanish. Since we deal with the dynamic graph of  a group, there is a path~$P'$ in~$\op{D}(G,S)$ that leads to the same vertex as~$P$ and begins with a path~$P''$ whose projection to the Cayley graph is a~cycle. By the centrality and additivity of the measure, $\nu(P'')\ge \nu(P') = \nu(P) > 0$. Applying the arguments from the proof\ft{Lemma~\ref{cor:Markov-time} deals with the case of a nondegenerate measure, but the arguments from its proof apply also to a degenerate measure provided that we consider paths on which the measure is concentrated.} of the homogeneity lemma~\ref{cor:Markov-time}, we deduce that the conditional measure on the subgraph~$\op{D}_v$ for the terminal vertex~$v$ of~$P''$ is isomorphic to the original measure~$\nu$. It follows that the original measure~$\nu$ takes a nonzero value not only on~$P''$, but also on all paths corresponding to powers of the cycle given by the projection of~$P''$. Thus, the original measure takes a~nonzero value on arbitrarily long paths that project to cycles. But it follows from standard results of the theory of Markov chains (see Theorems~3.2 and~3.3 in~\cite{ChKL64}) that for every path~$Q$ in the Cayley graph there is  $N\in{\mathbb N}$ such that for every $n>N$ the path~$Q$ can be extended to a cycle of length~$n$ provided that cycles of length~$n$ do exist. It remains to use the centrality of the measure to obtain a~contradiction with the assumed degeneracy.
\end{proof}

By the centrality of the measures under consideration, Theorem~\ref{lem:degenerate-simple} implies Corollary~\ref{cor:degenerate-almost-geodesic} which says that degenerate central measures are concentrated on paths whose projections to the Cayley graph are close to  \textit{geodesics}. For a rigorous description of the notion of ``being close to geodesics,'' we introduce the following notion of the \textit{defect} of a path.

\begin{definition}
The \textit{defect} of a finite path in the Cayley graph of a group is the difference between the length of this path and the length of the shortest path with the same beginning and end. The \textit{defect} of an infinite path is the supremum of the defects of its finite parts. Paths with zero defect are called \textit{geodesics}. The \textit{defect}  of a (finite or infinite) path in the dynamic graph of a group is the defect of its projection to the Cayley graph.
\end{definition}

\begin{corollary}\label{cor:degenerate-almost-geodesic}
Let $G$ be a finitely generated group and $S$ be a finite collection of generators of~$G$. Then there exists~$N\ge 0$ such that all degenerate central measures on the path space of the graph~$\op{D}(G,S)$ are concentrated on paths with defect at most~$N$.
\end{corollary}

\begin{proof}
Let $D$ be the greatest common divisor of all lengths of cycles in the Cayley graph of the pair~$(G,S)$. Then, as we know from the combinatorics of Markov chains (see Theorem~3.3 in~\cite{ChKL64}), there exists $J\in{\mathbb N}_0$ such that for every $j\ge J$ the Cayley graph contains a cycle of length~$jD$. Note also that the defect of every path is divisible by~$D$ (extend the path and the corresponding geodesic to cycles by the same path). Let $P$ be a finite path in the dynamic graph (that begins at the marked vertex) with defect $\operatorname{def}(P)>JD$. Let $Q$ be a geodesic path in the Cayley graph connecting the same vertices as the projection of~$P$. Then, since $\operatorname{def}(P)>JD$, the Cayley graph contains a cycle~$Z$ of length  $|Z|=\operatorname{def}(P)$. We may assume without loss of generality that $Z$ starts and ends at the identity of the group. Then the path consisting of the cycle~$Z$ and the path~$Q$ has the same length and connects the same vertices as the projection of~$P$. By the centrality of the measure, Theorem~\ref{lem:degenerate-simple} implies that degenerate measures vanish on the path~$P$. It remains to set $N:=JD$.
\end{proof}

It turns out that one can easily give sufficient conditions on~$G$ and~$S$,  not too restrictive,  under which all degenerate central measures are concentrated on paths whose projections to the Cayley graph are geodesics.

\begin{corollary}\label{cor:degenerate-geodesic}
Let $G$ be a finitely generated group and $S$ be a finite collection of generators of~$S$. Assume that the Cayley graph~$\Gamma$ of the pair~$(G,S)$ contains a~cycle with length equal to the greatest common divisor of all lengths of cycles in~$\Gamma$ (this condition is satisfied, for example, if $S$ contains the identity of the group, or if $S$ is symmetric and all cycles in~$\Gamma$ are even). Then all degenerate central measures on the path space of the graph~$\op{D}(G,S)$ are concentrated on paths whose projections to~$\Gamma$ are geodesics.
\end{corollary}

\begin{proof}
By the assumptions, we are in the case~$N=J=0$ from the proof of Corollary~\ref{cor:degenerate-almost-geodesic}.
\end{proof}

The spectrum of cases in which degenerate central measures are concentrated on geodesics is by no means exhausted by Corollary~\ref{cor:degenerate-geodesic}. For example, in a commutative group the above property holds for any finite collection of generators (see~\cite{VM18}). Nevertheless, the following example shows that it may happen that the geodesics are not sufficient to 
describe the degenerate part of the absolute: a degenerate ergodic central measure is neither  necessarily concentrated on paths with zero defect, no  necessarily homogeneous.

\smallskip{\bf Example.}
Consider the Baumslag--Solitar group $BS(2,1)$ (this is the group with two generators~$a$ and~$b$ and one relation $ab^2a^{-1}=b$) with the collection of generators $S:=\{a,b,a^{-1},b^{-1}\}$. One can easily check that three infinite paths corresponding to three infinite words
$$
abbba^{+\infty},\; baabba^{+\infty},\; bbab^{-1}a^{+\infty}
$$
constitute a tail equivalence class, hence the uniform measure on these three paths is (central and) ergodic. At the same time, as can also be easily checked, the projections of these three paths to the Cayley graph are not geodesics, and the transition probabilities of the corresponding Markov process at the point ${g=abb=ba}$  are different for the third and fourth step.

\section{Examples of describing the absolute}

\subhead{6.1. The absolute of free groups and homogeneous trees.}
Ignoring the natural order of exposition, we begin with a more complicated and illustrative example of computing the absolute of a free group, and then consider the problem of computing the absolute in the order of increasing complexity of the structure of the group. This example clearly reveals the meaningfulness of the transition from the Poisson--Furstenberg boundary to the absolute.

Consider the homogeneous tree~$T_{q+1}$ with all vertices having valence~$q+1$.  The set $\mathcal{H}_{\min}$ of minimal positive harmonic (i.e., invariant under the Laplace operator) functions on~$T_{q+1}$ coincides with the family of functions of the form~$q^{-h(v)}$, where $h(v)$ is an arbitrary function on~$T_{q+1}$. One can easily check that for every (real or complex) number~$\alpha$, the power $(q^{-h(v)})^\alpha=q^{-\alpha h(v)}$ of a minimal harmonic function~$q^{-h(v)}$ is an eigenfunction of the Laplace operator with eigenvalue
\begin{equation*}
s_\alpha=\frac{q^\alpha+q^{1-\alpha}}{q+1}\,.
\end{equation*}
Let us present a complete description of the absolute (Theorem~2.1 in~\cite{VM15}): the absolute of the free group with respect to the natural generators is the direct product of the boundary of the free group and an interval:
$$
\mathcal{A}(T_{q+1})=\partial T_{q+1}\times [1/2,1].
$$
To obtain the main part of the absolute, one should consider the same product with the semiopen interval $[1/2,1)$.

Let us explain this formula. A pair~$\omega\times r$ from $\partial T_{q+1}\times [1/2,1]$ is interpreted as the Markov measure on the dynamic graph of~$T_{q+1}$ corresponding  to the Markov measure on paths in the tree~$T_{q+1}$ (which is the Cayley graph of the free group) with the following parameters: the probability~$p(g,gw)$ of the transition from an arbitrary vertex~$g$ to the vertex~$gw$, where $w$ is a generator, is equal to~$r$ if the edge leading to~$\omega$ is labelled by~$w$, and $(1-r)/q$ otherwise.

It is natural to regard the probability~$r$, the eigenvalue corresponding to the eigenfunction, and the rate defined below as functions of the number~$\alpha$ parametrizing the eigenfunctions of the Laplace operator. In this case, we obtain three characteristics of the Markov chain:

--- $r_\alpha$ is the above probability,

--- $s_\alpha$ is the eigenvalue,

--- $v_\alpha:=2r_\alpha-1$ is the rate at which the point moves toward the chosen point at infinity~$\omega$. %???

They are given by the formulas
\begin{gather*}
s_\alpha=\frac{q^\alpha+q^{1-\alpha}}{q+1},\\
r_\alpha=\frac{q^\alpha}{q^\alpha+q^{1-\alpha}}=\frac{1}{1+q^{1-2\alpha}},\\
v_\alpha=\frac{2}{1+q^{1-2\alpha}}-1=\frac{1-q^{1-2\alpha}}{1+q^{1-2\alpha}}=
\frac{q^\alpha-q^{1-\alpha}}{q^\alpha+q^{1-\alpha}}\,.
\end{gather*}

One can easily use these formulas to deduce direct relationships between the parameters.

The table below shows the correspondences between the parameters at the critical points.

%https://www.overleaf.com/help/50-how-do-i-change-column-or-row-separation-in-latex-tables#.WvthRC9eMxg
\setlength{\tabcolsep}{2pt}
\renewcommand{\arraystretch}{1.6}

\begin{table}[ht]
\begin{center}
%\begin{tabsize}
\footnotesize
\begin{tabular}{ |c|c|c|c|c|c|c|c|c|c|c|c|c|c|c|c|c| } %~~\vphantom{$L^{L^L}$}
\hline
\tikz[align=center] \node {parameters\\ \VPH}; & & & \tikz[align=center] \node {simple\\ random \VPH \\ walk \VPH \Vqh }; & \tikz[align=center] \node {critical \Vph \\ point: \Vph \\ \VPH loss of \Vph \\ ergodicity \Vph}; & \tikz[align=center] \node {harmonic \VPH \\ measure \Vph \\ (values \Vph \\ of parameters) \Vph}; & & \\
\hline
$\alpha$ & $-\infty$ & \BDOTS & $0$ & $\frac12$ & $1$ & \BDOTS & $+\infty$ \\[2pt]
\hline
$r_\alpha$\VPH & $0$ & \BDOTS & $\frac{1}{q+1}$ & $\frac12$ & $\frac{q}{q+1}$ & \BDOTS & $1$ \\[2pt]
\hline
$v_\alpha$\VPH & $-1$  & \BDOTS & $\frac{1-q}{q+1}$ & $0$ & $\frac{q-1}{q+1}$ & \BDOTS & $1$ \\[2pt]
\hline
$s_\alpha$\VPH & $+\infty$ &\BDOTS & $1$ & $\frac{2\sqrt{q}}{q+1}$ & $1$ & \BDOTS & $+\infty$ \\[2pt]		
\hline
\end{tabular}
%\end{tabsize}
%\vspace{0pt}
%\caption{Зн чения п р метров в ключевых точк х}
\label{tbl:table}
\end{center}
\end{table}
%\vspace{-15pt}

The value $\alpha=1/2$ corresponds to the critical point at which the Markov chain loses ergodicity (phase transition). At this point, eigenfunctions stop being minimal. More exactly, for every eigenvalue greater than~$2\sqrt{q}/(q+1)$ there are two collections of eigenfunctions: one of them, corresponding to $\alpha>1/2$, consists of minimal eigenfunctions, while the other one consists of eigenfunctions that are not minimal. This distinction is well illustrated by the change of the rate~$v_\alpha$. For $\alpha>1/2$ (which is the same as $r_\alpha>1/2$), almost every trajectory of the random walk approaches the corresponding point of the boundary with a linear rate~$v_\alpha$ (drift). For  $\alpha=1/2=r_{1/2}$, there is no such convergence. Negative rates correspond to moving away from the point at infinity. 

The value $\alpha=1$ corresponds to ergodic Markov measures associated with minimal harmonic functions. The value $\alpha=0$ corresponds to the Markov chain generated by the Laplace operator; more exactly, the transition probabilities of this chain determine the corresponding Laplace operator. Its decomposition 
 into ergodic components determines exactly the harmonic measure on the fiber~$\partial T_{q+1}\times q/(q+1)$.

\smallskip
\subhead{6.2. The absolute of finite groups.}

\begin{theorem}
The absolute of every finite group for every choice of a collection of generators~$S$ consists of a single point, which is the uniform measure on the paths of equal length. The transition probabilities of the corresponding Markov process on the Cayley graph are equal to~$1/|S|$ for all edges of the graph. This measure belongs to the Laplacian part, the degenerate part of the absolute being empty.
\end{theorem}

\begin{proof}
Let us introduce the notion of a \textit{system of generators}: by a \textit{system of generators} of a group~$G$ we mean a (finite, in our case) set~$S$ and a mapping $S\to G$ (not necessarily an embedding) whose image generates~$G$ as a semigroup. It is natural to interpret systems of generators as collections of generators with multiplicities. All constructions of this paper can be obviously extended to the case of a group with a chosen system of generators: the notion of Cayley graph (in this case, it may have multiple edges), dynamic graph, central measure, absolute,~etc. For the one-element group with a system of generators of any cardinality, the whole simplex of central measures consists, obviously, of one measure described in the statement of the theorem. Considering the epimorphism from a finite group to the one-element group, we obtain, applying Theorem~\ref{th:ext-FC} (more exactly, its direct generalization to the case of collections of generators with multiplicities) that the Laplacian absolute of a finite group consists of this measure. The degenerate part of the absolute of a finite group is empty by Theorem~\ref{lem:degenerate-simple}, since every infinite path in a finite graph contains cycles.
\end{proof}

\subhead{6.3. The absolute of commutative groups.}
The following results are obtained in~\cite{VM18}.

\begin{theorem}[on the absolute of commutative groups and semigroups]
\label{th:Levyprocess-AM}
For every commutative semigroup and an arbitrary finite collection of generators, the set of ergodic central measures  {\rm(}i.e., the absolute{\rm)} coincides with the set of central measures that give rise to Markov chains with independent identically distributed increments. Thus, the absolute is in a bijective correspondence with the collection of all measures on the set of generators that determine Markov chains with this centrality property.
\end{theorem}

\begin{theorem}[on the topology of the absolute of commutative groups]\label{th:abel-top-group}
The absolute of a finitely generated commutative group with respect to any finite collection of generators is homeomorphic to a closed disk with dimension equal to the rank of the group. The Laplacian part of the absolute corresponds to the interior of the disk. 
\end{theorem}

\subhead{6.4. The absolute of the discrete Heisenberg group.}
The problem of describing the Laplacian part of the absolute of any finitely generated nilpotent group reduces to the  problem we have already solved of describing the Laplacian part of the absolute of its abelianization (Corollary~\ref{cor:nilpotent}). But the degenerate part of the absolute of a nilpotent group can differ significantly from the degenerate part of the absolute of its abelianization, and to describe it in the general case is a difficult problem. In a paper under preparation, we describe the structure of the absolute for the Heisenberg group with the standard collection of generators. In the following theorem we give a description of the absolute of the Heisenberg group from that paper in terms of measures on the space of infinite words in the alphabet of generators (this space~$S^\infty$ is in a bijective correspondence with the space of infinite paths in the dynamic graph).

\begin{theorem}\label{th:Heis}
{\rm1.} The absolute $\mathcal{A}(N_2,S)$ of the discrete Heisenberg group
\begin{equation*}
N_2=\langle x,y \mid [[x,y],x]=[[x,y],y]=1\rangle
\end{equation*}
with the collection of generators $S=\{x,y,x^{-1},y^{-1}\}$ is the union of a countable set of atomic measures and the set~$\mathcal{B}$ of Bernoulli measures~$\mu^\infty$ on~$S^\infty$ for which the generating measure~ $\mu$ on~$S$ satisfies the condition $\mu(x)\cdot\mu(x^{-1})=\mu(y)\cdot\mu(y^{-1})$.

{\rm2.} There is a natural bijection between the set~$\mathcal{C}$ of atomic measures from~$\mathcal{A}(N_2,S)$ and the set~$\mathcal{W}$ of words of the form $a^mb^nab^{+\infty}$, where $\{a,b\}$ is a pair of elements from~$S$ that are not inverse and $(m,n)\in{\mathbb N}_0\times{\mathbb N}_0$;  namely, for every measure~$\nu$ from~$\mathcal{C}$ there is a unique word~$W$ in~$\mathcal{W}$ such that $\nu$ is the uniform measure on the tail equivalence class of~$W$; the tail equivalence class in~$N_2$ of every word from~$\mathcal{W}$ is finite, and the uniform measure on this class belongs to~$\mathcal{C}$.
\end{theorem}

%\pagebreak

\renewcommand{\refname}{{\normalsize\sc References}}

\end{document}